\documentclass[twoside,a4paper,reqno,11pt]{amsart} 
\usepackage[top=28mm,right=30mm,bottom=28mm,left=30mm]{geometry}

\usepackage{amsfonts, amsmath, amssymb, mathrsfs, bm, latexsym, stmaryrd, array, hyperref, mathtools, bold-extra}

\usepackage{etaremune}
\usepackage{enumitem}

\newcommand{\GL}{\operatorname{GL}}
\newcommand{\gl}{\mathfrak{gl}} 
\newcommand{\SL}{\operatorname{SL}}
\newcommand{\sL}{\mathfrak{sl}} 
\newcommand{\SO}{\operatorname{SO}}

\makeatletter
\newcommand{\imod}[1]{\allowbreak\mkern4mu({\operator@font mod}\,\,#1)}
\makeatother

\newtheorem*{conj*}{Conjecture}

\newtheorem{thm}{Theorem}[section] 
 
\newtheorem{lem}[thm]{Lemma}
\newtheorem{cor}[thm]{Corollary}

\theoremstyle{definition}

\begin{document}

\title{Closures of Certain Matrix Varieties and Applications}

\author{William Chang}
\address{Department of Mathematics,  UCLA, 520 Portola Plaza, Los Angeles, CA 90095, USA}
\email{chang314@g.ucla.edu}

\author{Robert M. Guralnick}
\address{R.M. Guralnick, Department of Mathematics, University of Southern California, Los Angeles, CA 90089-2532, USA}
\email{guralnic@usc.edu}

\dedicatory{Dedicated to the memory of  Irina Suprunenko}  

\begin{abstract}  We prove some results about closures of certain matrix varieties consisting of elements with the
same centralizer dimension.  This generalizes a result of Dixmier and has applications to topological generation of
simple algebraic groups.  
 \end{abstract}

 

\date{\today}

\maketitle

\section{Introduction}\label{s:intro}

Let $k$ be an algebraically closed field of characteristic $p \ge 0$ and let $M_n(k)$ denote the algebra of $n \times n$ matrices over $k$.
 If $A \in M_n(k)$, we let 
$D(A)$ denote the subset of $M_n(k)$ consisting of all elements with the same Jordan canonical
form as $A$ up to changing the eigenvalues (but with the same number of distinct eigenvalues).

Note that $D(A)$ is closed  under conjugation by $\GL_n(k)$ and any two elements in $D(A)$ have
conjugate centralizers.   More generally, the fixed space of any two elements in $D(A)$ on the $d$th Grassmanian, $\mathcal{G}_d$
are also conjugate and in particular the dimension of the fixed space on $\mathcal{G}_d$ is constant on $D(A)$.

We generalize a result of Dixmier \cite{D} (for the case of type A).    Let $G$ be a simple algebraic group over an algebraically closed field.
A unipotent element $u \in G$ is called parabolic if there exists a parabolic subgroup $P$ such that $C_P(u)$ is an open dense subset of the
unipotent radical of $P$.   One similarly defines parabolic nilpotent elements in the Lie algebra of $G$.   Diximier  proved (in characteristic $0$)
that if $A$ is parabolic, then  $A$ is the limit of semisimple elements such that the dimensions of the  centralizers of each of the semisimple elements is the same as the dimension of
the centralizer of $A$.   This was used by Richardson \cite{Ri} to prove that the commuting variety of a reductive Lie algebra $\mathfrak{g}$ in characteristic $0$
is an irreducible variety of dimension equal to $\dim \mathfrak{g} + \mathrm{rank}(\mathfrak{g})$ (this was proved for $\gl_n$ by Motzkin and Taussky \cite{MT}
in all characteristics).    Levy \cite{Le} observed that Dixmier's result goes through in good characteristic as well and so Richardson's proof for the
irreducibility of the variety of commuting pairs goes through in good characteristic as well. 

We generalize Dixmier's result for $\gl_n$ by obtaining the same conclusion in arbitrary characteristic (using the Zariski topology)  and for arbitrary elements.   
 We also prove a sandwich result by showing that any such element is trapped between semisimple
and equipotent elements (i.e. having a single eigenvalue) all with the same centralizer dimension.  
We also need to verify that our semisimple and nilpotent elements have some extra properties (required for applications).
We combine this with a recent
result of Guralnick and Lawther \cite[Prop. 3.2.1]{GL} to conclude that all the fixed point spaces on all Grassmanians are the same
for elements in these subvarieties.   

This will be used in \cite{GG1} to deduce some results related to topological generation of simple algebraic groups   and
to extend the results of \cite{BGG} from semisimple and unipotent classes to all classes.   

Our main result is the following:

\begin{thm} \label{t:dixmier}   Let $A \in M_n(K)$.   Then there exist a semsimple element $S \in M_n(k)$ and an equipotent element $N \in M_n(k)$
such that the following hold:
\begin{enumerate}
\item  The Zariski closure of $D(S)$ contains $D(A)$;
\item  The Zariski closure of $D(A)$ contains $D(N)$;
\item  $A,S$ and $N$ all have centralizers of the same dimension; and
\item  If $1 \le d < n$,  then $A, S$ and $N$ all have fixed point spaces of the
same dimension on $\mathcal{G}_d$.
\end{enumerate}
\end{thm}

Note that in our proof controlling the determinant or trace of the elements is not an issue and so the same results hold for 
$\sL_n(k)$, $\GL_n(k)$ and $\SL_n(k)$.  We make some remarks in the last section regarding the symplectic and orthogonal cases. 

We also give some consequences regarding generation of simple algebraic groups that will be required in the sequel \cite{GG1}.
(3) was observed in \cite{GG} for the special case when $A$ is semisimple  (i.e. the existence of a nilpotent class with the
same centralizer dimension and with the largest eigenspace of the same dimension) and was  used to prove results about generic stabilizers. 
We give the proof in the next section and some applications in the following section.   

The first author thanks USC for support by a Provost's Undergraduate Research Grant.  The second author  was partially support by 
the NSF grant  DMS-1901595 and a Simons Foundation Fellowship 609771.  We thank the referee for their useful comments and
careful reading of the manuscript. 

\section{Proofs} 

  If $A \in M_n(k)$ has $m$ distinct eigenvalues, we set $\Delta(A)$ to be the set of $m$ partitions
$\Delta_i$ where $\Delta_i$ is the partition associated to the Jordan blocks of each
generalized eigenspace of $A$.

Let $X(\Delta) =\{ A \in M_n(k) | \Delta(A)=\Delta\}$ for $\Delta$ a set of $m$ partitions
whose sizes add up to $n$.   Note that if $A,B \in X(\Delta)$, then the centralizers of
$A$ and $B$ are conjugate and in particular have the same dimension.

Let $\Sigma\Delta = \sum \Delta_i$ where by the sum of partitons we mean the usual addition
(just view a partition as row vector with nonincreasing entries -- adding $0$'s to make the vectors have
the same length).        Given a partition $\Gamma$, let $\Gamma'$ be the transpose partition.

If $\Gamma$ is a partition of $n$,  let $U(\Gamma)$ be set of all matrices with a single eigenvalue
and the sizes of the Jordan blocks be given by $\Gamma$ and let $S(\Gamma)$ be the set
of semisimple matrices with the dimensions of the eigenspaces given by $\Gamma$. 

We first note the following elementary result.

\begin{lem} \label{l:basic}  Let $A$ be an upper triangular matrix with diagonal entries
contained in the set $\{a_1, \ldots a_s\}$ with the multiplicity of $a_i$ equal to $d_i$.
Assume that the entries in the $i, i+1$ positions are all nonzero.   Then  the Jordan canonical
form for $A$ consists of one Jordan block for each $a_i$ and it has size $d_i$.
\end{lem}

\begin{proof} Note that $A$ is regular (i.e. its centralizer has dimension $n$).   This follows
by noting $A$ is cyclic (i.e. the column vector $e_n = (0,0, \ldots, 1)$ generate the module
of column vectors for the algebra $k[A]$).   Thus, the characteristic and minimal polynomials
of $A$ are both $\Pi_{i=1}^s (x - a_i)^{d_i}$ and the result follows. 
\end{proof}  

Let   $A(a_1, \ldots, a_s; e_1, \ldots, e_s)$ denote the matrix above with the entries $i, i+1$ all equal to $1$
and all other entries (besides the diagonal) $0$.   Consider the $s$-dimensional affine space of all such matrices (as the $a_i$ range over all possibilities).
Note that the 
 generic points  in the variety of all such matrices (i.e. with $a_i \ne a_j$ for $i \ne j$) are all  regular.    If all the $a_i=a$, then
 the matrix is a regular equipotent matrix.

If $A \in M_n(k)$, let $\mathcal{G}_d^A$ be the fixed space of $A$ on $\mathcal{G}_d$ 
(i.e. the set of $d$-dimensional subspaces $W$ of $k^n$ with $AW \subseteq W$).

Our first result is the following:

\begin{thm} \label{t:unipotent}   Let $\Delta$ be as above.   Let $\Gamma = \Sigma\Delta$. 
\begin{enumerate}
\item The closure of $S(\Gamma')$ contains $X(\Delta)$.
\item   The closure of $X(\Delta)$ contains  $U(\Gamma)$.
\item  If $A \in S(\Gamma') \cup X(\Delta) \cup U(\Gamma)$, then
$\dim C(A) = \sum d_i^2$ where $\Gamma'$ is the partition $d_1 \ge d_2 \ge \ldots$.
\item   If $A \in S(\Gamma') \cup X(\Delta) \cup U(\Gamma)$, then $\dim \mathcal{G}_d^A$
is constant. 
\end{enumerate}
\end{thm}  

\begin{proof}   We prove (1). 
First suppose that $m=1$.  Then $X(\Delta)$ consists of elements with a single eigenvalue.
Then $\Gamma=\Delta$.

Let $m_1 \ge m_2 \ge \ldots \ge m_s$ be the pieces of the partition of $\Gamma$.    An element of 
$S(\Gamma')$ has has $m_1$ distinct eigenvalues and more generally has $m_j$ distinct eigenvalues
having multiplicity at least $j$.     Set $d=m_1$.

For any $a_1, \ldots, a_d \in k$,   let
  $B$ be the matrix with diagonal blocks of size $m_i$   with the $i$th diagonal block being $A(a_1, \ldots, a_{m_i}, 1, \ldots, 1)$. 
  
  Note that if the $a_i$ are distinct, then $B$ is semisimple and is in $S(\Gamma')$ and this is a generic point in the affine
  space of dimension $m$ obtained by allowing all possibilities for $a_i$.   In the closure are the elements when all $a_i$ are equal
  and so $X(\Gamma)$ is in the closure of $S(\Gamma')$ as claimed.
  
  In the general case,  we just choose  a $B_j$ as above corresponding to the Jordan blocks corresponding to the $j$th eigenvalue
  of an element in $X(\Delta)$ giving another copy of affine space.  The result follows by considering each block separately.

We prove (2) similarly.   Let $A \in X(\Delta)$.   First consider the case that 
each partition in $\Delta$ has just one part.   By taking generic elements and taking
elements in the closure with a single eigenvalue, we see we can obtain a regular element
with a single Jordan block (and any eigenvalue).   In general, we decompose $A$  into pieces
corresponding to $\Sigma\Delta=\Gamma$ with the $j$th piece having a single Jordan block corresponding
to each eigenvalue of size the $j$th part of $\Gamma$ and the closure contains elements with a single eigenvalue
with the partition of Jordan blocks corresponding to $\Gamma$.

The formulas for the dimension of centralizers and semisimple elements and nilpotent elements give
that the centralizers have the dimension for elements in $S(\Gamma')$ and $U(\Gamma)$.  This observation
together with (1) and (2) proves (3). 

The equality of the dimension of the fixed spaces on Grassmanian follows for the unipotent and semisimple elements
by \cite[Prop. 3.2.1]{GL}.   Then by (1) and (2),  (4) follows. 
\end{proof}

Clearly, the result (with essentially the same proof) holds with $M_n(k)$ replaced by either $\GL_n(k)$ or $\SL_n(k)$.

\section{An Application} 

We now apply Theorem \ref{t:unipotent}  to the action on Grassmanians.   
Recall that $A \in M_n(k)$ fixes a subspace $W$ means that $AW \subseteq W$.

Let $d(A)$ be the dimension of the largest eigenspace of $A$.  Observe that if $\Gamma=\Sigma\Delta$ with $\Delta=\Delta(A)$,
then $d(A) = d(B) = d(C)$ by Theorem \ref{t:unipotent}  applied to $\mathcal{G}_1$ (or by observation).

We now generalize the result \cite[Lemma 3.35]{BGG} which was stated for unipotent or semisimple elements.

\begin{thm} \label{t:sum}   Let  $A_1, \ldots, A_s \in  M_n(k)$.  Assume that $\sum d(A_i) \le (s-1)n$.
Then one of the following holds:
\begin{enumerate}
\item  Each $A_i$ has a quadratic minimal polynomial and $s=2$; or
\item  $\sum_i \dim \mathcal{G}_e^{A_i} < (s-1) \dim \mathcal{G}_e = (s-1)e(n-e)$ for $1 \le e  \le n/2$.
\end{enumerate}
\end{thm}

\begin{proof}   There is no harm is adding a scalar to each $A_i$ and so we may assume that the 
$A_i$ are all invertible.  This is proved in \cite[Lemma 3.35]{BGG} in the case each  $A_i$ is either semisimple or unipotent.
 The previous result shows that this implies the result for arbitrary $A_i$. 
\end{proof}

\begin{cor} \label{c:2-space}   Let $H$ be a closed irreducible subgroup of $\GL_n(k)$ that has a dense
orbit $\mathcal{O}$ on $\mathcal{G}_2$.   Let $x_1, \ldots, x_s \in H$ with $\sum d(x_i) \le (s-1)n$.   Assume moreover
that either $s > 2$ or $s=2$ and $x_1$ does not have a quadratic minimal polynomial.  Then
for generic $h_i \in H$,  $\langle x_1^{h_1}, \ldots, x_s^{h_s} \rangle$ do not fix a point of $\mathcal{O}$.
\end{cor}

\begin{proof}   It follows by the previous result that Theorem \ref{t:sum}(2) holds.   
Let $C_i$ be the $H$-conjugacy class of $x_i$.  
By considering $H$ acting on $\mathcal{O}$, 
it follows by \cite[Lemma 3.14]{BGG} that the subset of $C_1 \times C_2 \times \ldots \times C_s$
that have a fixed point on $\mathcal{O}$ is contained in a proper closed subvariety of 
$C_1 \times C_2 \times \ldots \times C_s$.
\end{proof}  

Note that this result holds for $H$ a symplectic or special  orthogonal group acting on nondegenerate 
$2$-spaces.   

There should be an analogous but more complicated result both for higher dimensional Grassmanians
and also for actions on the variety of totally singular spaces of a given dimension. 

We note that one can generalize Dixmier's result for arbitrary elements in the  symplectic and orthogonal Lie algebras in characteristic not $2$.
Indeed, any element that does not have $0$ as an eigenvalue is in the Zariski closure of the set of semisimple elements with the same
centralizer dimension.   There is a similar statement for the groups.   In good characteristic, Diximier's results still holds for parabolic nilpotent and unipotent
elements. 
  We do not require this for our application (even for groups of this type) since
these groups have dense open orbits on the Grassmanians (indeed, they have only finitely many orbits on each Grassmanian) and so
the result for $\gl$ is sufficient. 

\section{Declarations}

The authors declare no competing interests.  Both authors contributed equally.

\end{document}